\documentclass[letterpaper, bothsides, 12pt]{article}  
\usepackage[letterpaper,left=2.5cm,top=2.5cm,right=2.5cm,bottom=2.5cm]{geometry}  
\usepackage[utf8]{inputenc}
\usepackage{mathpazo} %une écriture italique
\usepackage{latexsym}  
\usepackage{layout}
\usepackage{float}
\usepackage{textcomp}
\usepackage{amstext}
\usepackage[pdftex]{graphicx}
\usepackage{multicol} 
\usepackage{longtable}  
\usepackage{lscape} 
\usepackage{multirow} % Tablas multifila
\usepackage{array} 
\usepackage{amssymb}
\usepackage{xspace}
\usepackage{amsmath} 
\usepackage{amsthm}
\usepackage[T1]{fontenc} 
\usepackage{listings} 
\usepackage{graphicx}
\usepackage{verbatim}
\def\R{\mathbb{R}}

\newcommand{\be}{\begin{equation}}
\newcommand{\ee}{\end{equation}}

\theoremstyle{plain}
\newtheorem{theorem}{Theorem}[section]
\newtheorem{definition}[theorem]{Definition}

\newtheorem{proposition}[theorem]{Proposition}
\newtheorem{lemma}[theorem]{Lemma}
\newtheorem{remark}[theorem]{Remark}

\numberwithin{theorem}{section}
\numberwithin{equation}{section}

\def\R{\mathbb{R}}

\usepackage{mathrsfs}

\usepackage[pdftex]{color}
\usepackage[pdftex,colorlinks=true,linkcolor=red,citecolor=blue,urlcolor=cyan]{hyperref}
\author{Emmanuel Wend-Benedo Zongo \footnote{Université Paris-Saclay,~~email:~emmanuel.zongo@universite-paris-saclay.fr}}
\date{}
\begin{document}
%\mainmatter 
\title{Bifurcation results for quasi-linear operators from the Fu\v cik spectrum of the Laplacian}
\maketitle
%and Bernhard Ruf \authormark{2}
%\tableofcontents
\vspace*{1.5cm}
 \begin{abstract}
In this paper, we analyze an eigenvalue problem for a quasi-linear elliptic operators involving Dirichlet boundary condition in an open smooth bounded set of $\R^N$. We investigate a bifurcation results (from trivial solution and from infinity) of an eigenvalue problem involving the $(p,2)$-Laplace operator, from the Fu\v cik spectrum of the Laplacian.
\end{abstract}
\tableofcontents
\section{Introduction}
 Assume $\Omega\subset\mathbb{R}^N$ ($N\geq 2$) is an open bounded domain with smooth boundary $\partial\Omega.$ A classical result in the theory of eigenvalue problems guarantees that the problem
\begin{equation}\label{dirichletlap}
\left\{
\begin{array}{l}
-\Delta u=\displaystyle \lambda  u~~\text{in $\Omega$},\\
u = \displaystyle 0~~~~~~~~\text{on $\partial\Omega$}
\end{array}
\right.
\end{equation}
possesses a nondecreasing sequence of eigenvalues $0<\lambda_1<\lambda_2\leq\dots$ and a sequence of corresponding eigenfunctions which define a Hilbert basis in $L^2(\Omega)$, (see \cite{Hen}). Moreover, it is known that the first eigenvalue of problem (\ref{dirichletlap}) is characterized in the variational point of view by,
$$
\lambda_1:=\inf_{u\in W^{1,2}_0(\Omega)\backslash\{0\}}\left\{
\frac{\int_{\Omega}|\nabla u|^2~dx}{\int_{\Omega}u^2~dx}\right\}
.$$ 
In \cite{ZB}, the authors investigated the asymptotic behavior of the spectrum and the existence of multiple solutions of the following nonlinear eigenvalue problem
\begin{equation}\label{equa2}
  \begin{cases}
    -\Delta_pu-\Delta u=\lambda u ~~~\text{on}~~\Omega,\\
    u=0~~~~~\text{on}~~\partial\Omega,
  \end{cases}
\end{equation}
where $-\Delta_pu=-\text{div}(|\nabla u|^{p-2}\nabla u)$ denotes the $p$-Laplace operator.  In \cite{ZB} it was shown that for $p>2$ there exist eigenvalue branches emanating from $(\lambda_k,0),$ and for $1<p<2$ emanating from $(\lambda_k,\infty)$. A nonlinear generalization of the spectrum of problem (\ref{dirichletlap}) is the Fu\v cik spectrum, given by the values $(\lambda_+,\lambda_-)\in\R^+\times\R^+$ for which the problem $-\Delta u=\lambda_+u^+-\lambda_-u^-$ with Dirichlet boundary condition has nontrivial solution. The values $\lambda_+$ and $\lambda_-$ which are such that $-\Delta u=\lambda_+u^+-\lambda_-u^-$ has nontrivial solution $u\in W^{1,2}_0(\Omega)$ will be called half-eigenvalues, while the corresponding solutions $u$ will be called half-eigenfunctions.
The Fu\v cik spectrum was introduced by S. Fu\v cik  \cite{SF} and N. Dancer \cite{ND} in the 70's, mainly motivated by the works of A. Ambrosetti and G. Prodi \cite{AG} in connection with what is known in the literature as ``asymmetric problems".\\
\indent In this paper we consider the following nonlinear eigenvalue problem
\begin{equation}\label{equa3}
  \begin{cases}
    -\Delta_pu-\Delta u=\lambda_+u^+-\lambda_-u^- ~~~\text{on}~~\Omega,\\
    u=0~~~~~\text{on}~~\partial\Omega.
  \end{cases}
\end{equation}
%where $\lambda_+,\lambda_-\in\R$ are values such that $-\Delta u=\lambda_+u^+-\lambda_-u^-$ has nontrivial solutions. This stands as a generalization of the spectrum of equation (\ref{equa2}). The problem under study has some relation to the so-called Fu\v cik spectrum. This is explain in section \ref{S1}.
The operator $-\Delta_p-\Delta$ appears in quantum field theory  (see, \cite{Ben}), where it arises in the mathematical description of propagation phenomena of solitary waves. The main purpose of this paper is to study the asymptotic behavior of the spectrum of problem (\ref{equa3}) for all $p>2$ and $1<p<2.$ We intend to study problem (\ref{equa3}) as a bifurcation in the Fu\v cik eigenvalues from trivial solution if $p>2$ and a bifurcation from infinity if $1<p<2$.\\
\indent In \cite{RB}, the author has proved that the equation
\begin{equation}\label{ruf}
    Au-\gamma u^-=\lambda u+N(\lambda,u)
\end{equation}
has global bifurcation branches emanating from the points $(\lambda^1_k, 0)$ and $(\lambda^2_k, 0)$ in $\R\times L^2(\Omega)$ for any $k\geq 1.$ In equation (\ref{ruf}), the operator $A: D(A)\subset L^2(\Omega)\rightarrow L^2(\Omega)$ is an unbounded self-adjoint operator with compact resolvent and with eigenvalues $\lambda_1\leq \lambda_2\leq\dots,$ repeated according to their multiplicity, and $\gamma\in \R^+$ satisfies $0<\gamma<\min\{\lambda_k-\lambda_{k-1}, \lambda_{k+1}-\lambda_k\}$ ($k\geq 2$) and $N(\lambda,u)=o(\|u\|)$ for $u$ near zero uniformly on bounded $\lambda$ intervals.  The points $\lambda^1_k$ and $\lambda^2_k$ are called split eigenvalues of the operator $A-\gamma(\cdot)^-$ and their existence was proved by B. Ruf in \cite{Ruf}. Our work is inspired by the work of \cite{RB}.  In the present article, the point $(\lambda^1_k, 0)$ will be a bifurcation point for positive solutions and the point $(\lambda^2_k, 0)$ will be a bifurcation point for negative solutions of problem (\ref{equa3}). 
\\
\indent A similar structure as (\ref{ruf}) was proved in \cite{ES} concerning bifurcation problems for Pucci's operators. More precisely in \cite{ES}, the authors studied the following problem as a bifurcation problem from the trivial solution
\begin{equation}\label{pucci}
  \begin{cases}
    -\mathcal{M}^+_{\lambda,\Lambda}(D^2)u=\mu u+f(u) ~~~\text{on}~~\Omega,\\
    u=0~~~~~\text{on}~~\partial\Omega,
  \end{cases}
\end{equation}
(respectively $\mathcal{M}^-_{\lambda,\Lambda}$) where $\Omega$ is a bounded regular domain, and $\mathcal{M}^{\pm}_{\lambda,\Lambda}$ are the extremal Pucci's operators with parameters $0<\lambda\leq\Lambda.$ It is shown that bifurcation occurs form the ``half-eigenvalues" $\mu_k^+$ and $\mu_k^-$.\\ 
\\
Our main results are: let $\gamma\in\R^+$ that satisfies $0<\gamma<\min\{\lambda_k-\lambda_{k-1},\lambda_{k+1}-\lambda_k\},$ $k\geq 2$
\begin{itemize}
    \item For $p>2,$ we obtain bifurcation branches emanating from the points $(\lambda^1_k,0)$ and $(\lambda^2_k,0)$ in $\R\times L^2(\Omega)$ for the operator $-\Delta_p-\Delta-\gamma(\cdot)^-,$ where $\lambda^1_k$ and $\lambda^2_k$ are the split eigenvalues of $-\Delta-\gamma(\cdot)^-$,
    \item For $1<p<2,$ we obtain bifurcation branches emanating from the points $(\lambda^1_k,\infty)$ and $(\lambda^2_k,\infty)$ in $\R\times L^2(\Omega)$ for the operator $-\Delta_p-\Delta-\gamma(\cdot)^-$.
\end{itemize}
\indent This article is organized as follows. In section \ref{S1}, we present the Fu\v cik spectrum of problem (\ref{equa3}) and in section \ref{S2} we show that are solutions branches of problem (\ref{equa3}) that bifurcate from the Fu\v cik eigenvalues of problem (\ref{nonlinear4}). In section \ref{Sinf}, we discuss bifurcation from infinity of problem (\ref{equa3}).
For a further discussion on this topic, we refer the interested reader to the work of \cite{ES} and the references therein.\\
\\
We recall that if $1<p<q,$ then $L^q(\Omega)\subset L^p(\Omega)$ and as a consequence, one has $W^{1,q}_0(\Omega)\subset W^{1,p}_0(\Omega).$ In what follows, we denote by $\|.\|_{1,p}$ and $\|.\|_2$ the norms on $W^{1,p}_0(\Omega)$ and $L^2(\Omega)$  defined respectively by $$\|u\|_{1,p}=\left(\int_{\Omega}|\nabla u|^p~dx\right)^{\frac{1}{p}}~~\text{and}~~\|u\|_2=\left(\int_{\Omega}|u|^2~dx\right)^{\frac{1}{2}}.$$
\section{Fu\v cik spectrum and nonlinear spectrum}\label{S1}
An eigenvalue $\lambda\in\mathbb{R}$ of the quasi-linear elliptic operator $-\Delta_p-\Delta$ on $\Omega$ is such that the following problem 
\begin{equation}\label{nonlinear1}
  \begin{cases}
    -\Delta_pu-\Delta u=\lambda u ~~~\text{on}~~\Omega,\\
    u=0~~~~~\text{on}~~\partial\Omega
  \end{cases}
\end{equation}
has a nontrivial solution for  $p\in(1,\infty)\backslash\{2\}$. We denote by $\Sigma$ the spectrum of problem (\ref{nonlinear1}).
The Fu\v cik spectrum of  $-\Delta$  on $\Omega$ with Dirichlet boundary conditions is defined as the set $\mathcal{S}_{(\lambda_+,\lambda_-)}$ of pairs $(\lambda_+,\lambda_-)\in\mathbb{R}^2$ such that the problem 
\begin{equation}\label{nonlinear4}
  \begin{cases}
    -\Delta u=\lambda_+ u^+-\lambda_-  u^- ~~~\text{on}~~\Omega,\\
    u=0~~~~~\text{on}~~\partial\Omega
  \end{cases}
\end{equation}
has a nontrivial solution. Then,
the Fu\v cik spectrum of  $-\Delta_p-\Delta$  on $\Omega$ with Dirichlet boundary conditions is defined as the set $\Sigma_{(\lambda_+,\lambda_-)}$ of pairs $(\lambda_+,\lambda_-)\in\mathbb{R}^2$ such that the problem 
\begin{equation}\label{nonlinear2}
  \begin{cases}
    -\Delta_pu-\Delta u=\lambda_+u^+-\lambda_-u^- ~~~\text{on}~~\Omega,\\
    u=0~~~~~\text{on}~~\partial\Omega
  \end{cases}
\end{equation}
has a nontrivial solution for $p\in(1,\infty)\backslash\{2\}.$ Clearly $\Sigma_{(\lambda_+,\lambda_-)}$ generalizes the notion of the spectrum of problem (\ref{nonlinear1}), since $\Sigma$ consists of all values $\lambda\in\R$ such that problem (\ref{nonlinear2}) for $\lambda_+=\lambda_-=\lambda$ has a nontrivial solution.
Here we write $u^{\pm}=\max\{\pm u,0\}$ in $\Omega,$ where $u^+$ and $u^-$ denote respectively the positive and the negative part of $u.$ We also recall that $u=u^+-u^-.$\\
%We recall that a real value $\lambda$ is called an eigenvalue of $-\Delta$ on $\Omega$ with Dirichlet boundary conditions if the problem
%\begin{equation}\label{nonlinear3}
  %\begin{cases}
   % -\Delta u=\lambda u ~~~\text{on}~~\Omega,\\
   % u=0~~~~~\text{on}~~\partial\Omega
  %\end{cases}
%\end{equation}
%has a nontrivial solution.\\
%\bigskip

\subsection{Spectrum of problem (\ref{nonlinear2})}
\begin{definition}
We say that $u\in W^{1,p}_0(\Omega)$ $($if $p>2$ $)$ or $u\in W^{1,2}_0(\Omega)$ $($if $1<p<2$ $)$ with $u\not\equiv 0, \lambda_{\pm}\in\R$, is a weak solution of (\ref{nonlinear2}) if and only if 
\begin{equation}\label{weakform}
    \int_{\Omega}|\nabla u|^{p-2}\nabla u\cdot \nabla\varphi~dx+\int_{\Omega}\nabla u\cdot \nabla\varphi~dx=\lambda_+\int_{\Omega}u^+\varphi~dx-\lambda_-\int_{\Omega}u^-\varphi~dx
\end{equation}
for all $\varphi\in W^{1,p}_0(\Omega)$ $($if $p>2$ $)$, $\varphi\in W^{1,2}_0(\Omega)$ $($if $1<p<2$ $).$ 
\end{definition}
The function $u$ is called an half-eigenfunction associated with the pair $(\lambda_+,\lambda_-)\in\mathbb{R}^2.$
\begin{remark}
For $\lambda_+\leq \lambda_1$ and $\lambda_-\leq \lambda_1$, equation (\ref{nonlinear2}) has only zero solution.
%, where $\lambda_1^D$ stands as the first eigenvalue of problem (\ref{nonlinear3})
\end{remark}
%\subsection{The Fu\v cik spectrum of problem (\ref{nonlinear2})}
% Here we will describe the fucik spectrum
\subsection{Invertibility property}
\begin{lemma}\label{invertiblity}
Let $p>2.$ Then operator $-\Delta_p-\Delta-\gamma(\cdot)^-$ is a global homeomorphism between $W^{1,p}_0(\Omega)$ and its dual space for a fixed $\gamma\geq 0.$
\end{lemma}
For the proof of this lemma, we need the following lemma (see \cite{PL} for a proof).
\begin{lemma}\label{plindvist}
Let $p>2$. Then there exist two positive constants $c_1, c_2$ such that, for all $x_1,x_2\in\mathbb{R}^n,$ we have :
\begin{enumerate}
\item[(a)] $(x_2-x_1)\cdot (|x_2|^{p-2}x_2-|x_1|^{p-2}x_1)\geq c_1|x_2-x_1|^p$
\item[(b)]$\left||x_2|^{p-2}x_2-|x_1|^{p-2}x_1\right|\leq c_2(|x_2|+|x_1|)^{p-2}|x_2-x_1|$
\end{enumerate}
\end{lemma}
\begin{proof}[\textup{Proof of Lemma \ref{invertiblity}}]
Define the quasi-linear operator $\mathcal{M}:W^{1,p}_0(\Omega)\rightarrow W^{-1,p'}(\Omega)$ (with $1/p+1/p'=1$) by\\ $\langle \mathcal{M}u,v\rangle=\displaystyle{\int_{\Omega}}\nabla u\cdot\nabla v~dx+\displaystyle{\int_{\Omega}}|\nabla u|^{p-2}\nabla u\cdot \nabla v~dx-\gamma \int_{\Omega}u^-vdx$ for all $u,v\in W^{1,p}_0(\Omega).$\\
To show that $-\Delta_p-\Delta-\gamma(\cdot)^-$ is a homeomorphism, it is enough to show that $\mathcal{M}$ is a continuous strongly monotone operator, (see~\cite[Corollary 2.5.10]{KC}).\\
\noindent For $p>2,$ for all $u, v\in W^{1,p}_0(\Omega)$, by $(a)$, we get 
\begin{align*}
    \langle \mathcal{M}u-\mathcal{M}v,u-v\rangle &= \int_{\Omega}|\nabla(u-v)|^2dx+\int_{\Omega}\left(|\nabla u|^{p-2}\nabla u-|\nabla v|^{p-2}\nabla v\right)\cdot\nabla(u-v)~dx\\
    &-\gamma\int_{\Omega}(u^--v^-)(u-v)dx\\
&\geq \int_{\Omega}|\nabla(u-v)|^2dx+c_1\int_{\Omega}|\nabla(u-v)|^pdx-\gamma\int_{\Omega}(u^--v^-)(u-v)dx\\
&\geq  \int_{\Omega}|\nabla(u-v)|^2dx+c_1\int_{\Omega}|\nabla(u-v)|^pdx\\
&\geq  c_1\|u-v\|^p_{1,p}.
\end{align*}
Thus $\mathcal{M}$ is a strongly monotone operator. We claim that $\mathcal{M}$ is a continuous operator from $W^{1,p}_0(\Omega)$ to $W^{-1,p'}(\Omega).$  Indeed, assume that $u_n\rightarrow u$ in $W^{1,p}_0(\Omega).$ We have to show that $\|\mathcal{M}u_n-\mathcal{M}u\|_{W^{-1,p'}(\Omega)}\rightarrow 0$ as $n\rightarrow\infty.$ Using $(b)$, H\"older's inequality and the Sobolev embedding theorem, one has
\begin{eqnarray*}
\left|\langle \mathcal{M}u_n-\mathcal{M}u,v\rangle\right| &\leq & \int_{\Omega}\left||\nabla u_n|^{p-2}\nabla u_n-|\nabla u|^{p-2}\nabla u\right||\nabla v|dx+\int_{\Omega}|\nabla(u_n-u)||\nabla v|dx\\ &+&\int_{\Omega}\gamma|u_n^--u^-||v|dx\\
&\leq& c_2 \int_{\Omega}\left(|\nabla u_n|+|\nabla u|\right)^{p-2}|\nabla(u_n-u)||\nabla v|dx+\int_{\Omega}|\nabla(u_n-u)||\nabla v|dx\\ &+&\int_{\Omega}\gamma|u_n-u||v|dx\\
&\leq& c_2\left(\int_{\Omega}\left(|\nabla u_n|+|\nabla u|\right)^{p}dx\right)^{(p-2)/p}\left(\int_{\Omega}|\nabla(u_n-u)|^p dx\right)^{1/p}\left(\int_{\Omega}|\nabla v|^p dx\right)^{1/p}\\ &+& c_3\|u_n-u\|_{1,2}\| v\|_{1,2}\\
&\leq & c_4(\|u_n\|_{1,p}+\|u\|_{1,p})^{p-2}\|u_n-u\|_{1,p}\|v\|_{1,p}+c_5\|u_n-u\|_{1,p}\| v\|_{1,p}.
\end{eqnarray*}
Thus $\|\mathcal{M}u_n-\mathcal{M}u\|_{W^{-1,p'}(\Omega)}\rightarrow 0$, as $n\rightarrow +\infty,$ and hence $\mathcal{M}$ is a homeomorphism.
\end{proof}
\section{Bifurcation results}\label{S2}
Here we show bifurcation results from trivial solutions and from infinity of problem (\ref{nonlinear1}) for $p>2$ and $1<p<2$ respectively.
\subsection{Preliminary results toward the bifurcation results}
Let $p>2$, then for $u\in W^{1,p}_0(\Omega),$ we have $$-\Delta_pu-\Delta u=\lambda_+u^+-\lambda_-u^-\Leftrightarrow -\Delta_pu-\Delta u-(\lambda_+-\lambda_-)u^-=\lambda_+u.$$ We fix $\gamma=\lambda_+-\lambda_->0$ and we consider $\lambda:=\lambda_+$ as a parameter of the equation
\begin{equation}\label{bifurcation1}
  \begin{cases}
    -\Delta_pu-\Delta u-\gamma u^-=\lambda u ~~~\text{on}~~\Omega,\\
    u=0~~~~~~~~~~~\text{on}~~\partial\Omega.
  \end{cases}
\end{equation}
Using Lemma \ref{invertiblity}, we write equation (\ref{bifurcation1}) in the following form:
\begin{equation}\label{bifurcation2}
    u=\lambda(-\Delta)^{-1}u+\left\{(-\Delta_p-\Delta-\gamma(\cdot)^-)^{-1}-(-\Delta)^{-1}\right\}(\lambda u),
\end{equation}
where we consider
$$(-\Delta_p-\Delta-\gamma(\cdot)^-)^{-1}: L^2(\Omega)\subset W^{-1,p'}(\Omega)\rightarrow W^{1,p}_0(\Omega)\subset\subset L^2(\Omega),$$
and 
$$(-\Delta)^{-1}: L^2(\Omega)\subset W^{-1,2}(\Omega)\rightarrow W^{1,2}_0(\Omega)\subset\subset L^2(\Omega).$$
So, for $p>2$ the mapping $$(-\Delta_p-\Delta-\gamma(\cdot)^-)^{-1}-(-\Delta)^{-1}: L^2(\Omega)\subset W^{-1,p'}(\Omega)\rightarrow W^{1,p}_0(\Omega)\subset\subset L^2(\Omega)$$ is a compact mapping thanks to the Rellich-Kondrachov theorem. For simplicity, we write equation (\ref{bifurcation2}) as follows
\begin{equation}\label{bifurcation3}
  u=\lambda Au+T_{\lambda}(u),  
\end{equation}
with $Au=(-\Delta)^{-1}$ and $T_{\lambda}(u)=\left\{(-\Delta_p-\Delta-\gamma(\cdot)^-)^{-1}-(-\Delta)^{-1}\right\}(\lambda u).$\\
\\
The $L^2(\Omega)$ space with norm $\|\cdot\|_2$ admits the orthogonal decomposition 
\begin{equation}\label{bifurcation4}
    L^2(\Omega)=[e_k]\oplus [e_k]^{\perp},
\end{equation}
where $e_k$ is the normalized $k$-th eigenfunction associated to a simple eigenvalue $\lambda_k$ of the Dirichlet Laplace operator  (see \cite{Ruf}), and $[e_k]$ denotes the eigenspace of $\lambda_k.$ In view of (\ref{bifurcation4}), we set $H_k=[e_k]^{\perp}$ and $E_k=[e_k]$ for simplicity. An element of $u\in L^2(\Omega)$ in (\ref{bifurcation4}) will be seen as $u=\alpha e_k+v,$ with $\alpha\in\R$ and $v\in H_k.$\\
\\
The following result was obtained by B. Ruf in \cite{Ruf}. If $\gamma\geq 0$ satisfies 
\begin{equation}\label{bifurcation5}
  0<\gamma<\min\{\lambda_k-\lambda_{k-1}, \lambda_{k+1}-\lambda_k\},  
\end{equation}
then the following nonlinear eigenvalue problem
\begin{equation*}
    -\Delta u-\gamma u^-=\lambda u,
\end{equation*}
has a unique split eigenvalue $(\lambda^1_k,\lambda^2_k)\in [\lambda_k, \lambda_{k+1}]\times [\lambda_k, \lambda_{k+1}]$ of $-\Delta-\gamma(\cdot)^-$ and uniquely determined eigenfunctions $v^1_k, v^2_k $ in $L^2(\Omega)$ such that
$$-\Delta v^i_k-\gamma(v^i_k)^-=\lambda^i_kv^i_k,~~~~i=1,2.$$
For $\gamma=0$ the $k$-th split eigenvalue reduces to the eigenvalue $\lambda_k$ and $v^1_k, v^2_k$ to the eigenfunctions $e_k$, $-e_k$ respectively.\\
\\
Consider
\begin{equation}\label{preli1}
  \begin{cases}
    -\Delta u-\gamma u^-=\lambda u ~~~\text{on}~~\Omega,\\
    u=0~~~~~~~~~~~\text{on}~~\partial\Omega,~~~~\text{with}~~\gamma>0.
  \end{cases}
\end{equation}
We say that $u\in W^{1,2}_0(\Omega)$ with $u\not\equiv 0$ is a weak solution of (\ref{preli1}) if and only if 
\begin{equation}\label{preli2}
    \int_{\Omega}\nabla u\cdot \nabla v ~dx-\gamma\int_{\Omega}u^-v~dx=\lambda\int_{\Omega}uv~dx,
\end{equation}
for all $v\in W^{1,2}_0(\Omega).$
\begin{remark}
Suppose that $u$ is an eigenfunction associated to $\lambda$ in (\ref{preli2}). Then $\alpha u$ is also an eigenfunction associated to $\lambda$ for $\alpha>0.$
\end{remark}
\begin{comment}
Indeed, we have 
$$\int_{\Omega}\nabla(\alpha u)\cdot \nabla v ~dx-\gamma\int_{\Omega}(\alpha u)^-v~dx=\lambda\int_{\Omega}(\alpha u)v~dx$$
$$\Updownarrow$$
$$\alpha\int_{\Omega}\nabla u\cdot \nabla v ~dx-\gamma\alpha\int_{\Omega}u^-v~dx=\lambda\alpha\int_{\Omega} uv~dx~~~~\text{if}~~\alpha>0,$$
since $(\alpha u)^-=\frac{1}{2}(|\alpha u|-\alpha
u)=\frac{1}{2}(|\alpha|| u|-\alpha
u)=\alpha u^-,$ and this shows that $\alpha u$ is also an eigenfunction associated to $\lambda.$
We clearly see that if $\alpha<0,$ $(\alpha u)^-=-\alpha u^+$ and $\alpha u$ is not an eigenfunction associated to $\lambda.$\\
\\
\end{comment}
From this remark, we aim to show that problem (\ref{bifurcation1}) has bifurcation branches emanating from the points $(\lambda^1_k, 0)$ and $(\lambda^2_k,0)$ in $\R\times L^2(\Omega)$ in the direction $\alpha>0$ when $p>2.$\\
%Let $P_k$ and $Q_k$ denote respectively the orthogonal projections on $H_k$ and $E_k.$
In what follows, we assume that (\ref{bifurcation5}) holds. Let $\eta:=\frac{1}{2}\min\{|(e_k, v^1_k(\gamma))_2|, |(e_k, v^2_k(\gamma))_2|\}>0,$  and define for $\rho\in\R,$ $\rho>0,$
%$$\mathcal{K}_{\rho,\eta}(\gamma):=\{(\lambda,u)\in\R\times L^2(\Omega):~|\lambda-\lambda^1_k(\gamma)|<\rho,~~|(e_k,u)_2|>\eta \|u\|_2\},$$ 
$$\pmb{K^+_{\rho,\eta}}(\gamma):=\{(\lambda,u)\in\R\times L^2(\Omega):~|\lambda-\lambda^1_k(\gamma)|<\rho,~~(e_k,u)_2>\eta \|u\|_2\},$$
$$\pmb{K^-_{\rho,\eta}}(\gamma):=\{(\lambda,u)\in\R\times L^2(\Omega):~|\lambda-\lambda^2_k(\gamma)|<\rho,~~(e_k,u)_2<-\eta \|u\|_2\},$$
where $(\cdot,\cdot)_2$ denotes the $L^2(\Omega)$ inner product. Let
$$B_{\rho}:=\{(\lambda,u)\in\R\times L^2(\Omega):~(|\lambda|^2+\|u\|^2_2)^{1/2}<\rho\}.$$ 
%One sees that $\mathcal{K}_{\rho,\eta}(\gamma)$ is an open subset of $\R\times L^2(\Omega)$ consisting of two disjoint components $\mathcal{K}^+_{\rho,\eta}(\gamma)$ and $\mathcal{K}^-_{\rho,\eta}(\gamma)$ with 

 We see that both $\pmb{K^+_{\rho,\eta}}(\gamma)$ and $\pmb{K}^-_{\rho,\eta}(\gamma)$ are two disjoints convex cones.
 %$\pmb{K^-_{\rho,\eta}}(\gamma)=-\pmb{K^+_{\rho,\eta}}(\gamma)$ and $\pm t e_k\in \pmb{K^{\pm}_{\rho,\eta}}(\gamma).$\\
\\
For $\gamma>0$ fixed, let $\Sigma_{\gamma}$ denote the closure of the set of nontrivial solutions of problem (\ref{bifurcation1}). The following lemma localizes the possible solutions of equation (\ref{bifurcation1}). We refer to \cite[Lemma 1.24]{PH} for similar results.
\begin{lemma}\label{l1}
There exists $\rho_0>0$ such that for all $0<\rho<\rho_0:$
$$\left(\Sigma_{\gamma}\backslash\{(\lambda^1_k,0)\}\cap B_{\rho}\right)\subset \pmb{K^+_{\rho,\eta}}(\gamma).$$ If $(\lambda, u)\in \left(\Sigma_{\gamma}\backslash\{(\lambda^1_k,0)\}\cap B_{\rho}\right),$ then $u=\alpha e_k+v,$ where $|\alpha|>\eta \|u\|_2,$ and $|\lambda-\lambda^1_k|=o(1), v=o(\alpha)$ for $\alpha$ near zero. Moreover $\alpha>0.$ The same result holds for 
$\pmb{K^-_{\rho,\eta}}(\gamma).$
\end{lemma}
\begin{proof}
Suppose by contradiction that there is no $\rho_0$ as in Lemma \ref{l1}. Then there exist two sequences $\rho_n$ and $(\lambda_n, u_n)\in \left(\Sigma_{\gamma}\backslash\{(\lambda^1_k,0)\}\cap B_{\rho_n}\right)$ such that $\rho_n\rightarrow 0,$ $|\lambda_n-\lambda^1_k|\leq \rho_n$, $u_n\rightarrow 0$ and $|(e_k,u_n)_2|\leq\eta \|u_n\|_2 $ as $n\rightarrow\infty.$ Since $u_n\not\equiv 0,$ we introduce the following change of variable $\Bar{u}_n=u_n/\|u_n\|_2$ in the following equation
$$-\Delta_p u_n-\Delta u_n-\gamma(u_n)^-=\lambda_n u_n.$$ Then it follows that $-\|u_n\|^{p-1}_2\Delta_p\bar{u}_n-\Delta\bar{u}_n-\gamma(\bar{u}_n)^-=\lambda_n\bar{u}_n,$ for $p>2.$ Let $\mu\in\varrho(-\Delta)$ (the resolvent set of $-\Delta$), hence
$$\bar{u}_n=(-\Delta-\mu)^{-1}[(\lambda_n-\mu)\bar{u}_n+\gamma(\bar{u}_n)^-+\|u_n\|^{p-1}_2\Delta_p\bar{u}_n].$$ Using the fact that $(-\Delta-\mu)^-$ is compact, then for a subsequence, $\bar{u}_n\rightarrow \bar{u}$ in $L^2(\Omega)$ as $n\rightarrow\infty,$ and the with the continuity of $\Delta_p$ we find that 
$$\bar{u}=(-\Delta-\mu)^{-1}[(\lambda^1_k-\mu)\bar{u}+\gamma(\bar{u})^-],$$ since $\lambda_n\rightarrow\lambda^1_k$ and $u_n\rightarrow 0$ as $n\rightarrow\infty.$ This implies that $\bar{u}=v^1_k$ (or possibly $\bar{u}=v^2_k$ if $\lambda^1_k=\lambda^2_k$). Thus we get $2\eta\leq |(e_k,\bar{u})|\leq\eta$ since $|(e_k,u_n)_2|\leq\eta \|u_n\|_2 .$ This is a contradiction. Hence there exists $\rho_0$ as above and for $0<\rho<\rho_0,$ if $(\lambda, u)\in \left(\Sigma_{\gamma}\backslash\{(\lambda^1_k,0)\}\cap B_{\rho}\right),$ then $u=\alpha e_k+v$ with $|\alpha|>\eta \|u\|_2.$ We have $\|v\|_2\leq |\alpha|+\|u\|_2\leq |\alpha|+|\alpha|/\eta$ and $|\lambda-\lambda^1_k|<\rho\rightarrow0$. So, for $\alpha$ near zero, we have $|\lambda-\lambda^1_k|=o(1)$ and $v=o(\alpha).$
\end{proof}
\begin{remark}
The $\rho_0$ in Lemma \ref{l1} can be chosen indepently of $\gamma\in [0,\bar{\gamma}]$, where $\bar{\gamma}$ satisfies (\ref{bifurcation5}) if we set
$$\eta:=\frac{1}{2}\inf\limits_{\gamma\in [0,\bar{\gamma}]}\{|(e_k, v^1_k(\gamma))_2|, |(e_k, v^2_k(\gamma))_2|\}.$$
\end{remark}
We employ a Lyapunov-Schmidt method, that is, we write equation (\ref{bifurcation3}) as equivalent system in $\R\times H_k$ and $\R\times E_k.$ We denote by $P_k$ and $Q_k$ the orthogonal projections onto $H_k$ and $E_k$ respectively, that is $P_k(\alpha e_k+v)=v$ and $Q_k(\alpha e_k+v)=\alpha e_k.$ We write equation (\ref{bifurcation3}) in $\R\times H_k$ and $\R\times E_k$ as 
\begin{equation}\label{ls1}
  \begin{cases}
    v=\lambda Av +P_k(T_{\lambda}(\alpha e_k+v)),\\
    \alpha e_k=\lambda\alpha Ae_k+Q_k(T_{\lambda}(\alpha e_k+v)).
  \end{cases}
\end{equation}
This is equivalent to 
\begin{equation}\label{ls2}
    v=\lambda Av +P_k(T_{\lambda}(\alpha e_k+v))
\end{equation}
and 
\begin{equation}\label{lss2}
    \alpha=\frac{\alpha\lambda}{\lambda_k}+(Q_k(T_{\lambda}(\alpha e_k+v)), e_k)_2.
\end{equation}
We set now, 
$$R_{\gamma}(\alpha,\lambda,v)=\frac{\alpha\lambda}{\lambda_k}+(Q_k(T_{\lambda}(\alpha e_k+v)), e_k)_2-\alpha +\lambda,$$
\begin{equation}\label{lss3}
    S_{\gamma}(\alpha,\lambda,v)=(\lambda-\lambda^1_k) Av +P_k(\tilde{T}_{\lambda}(\alpha e_k+v))
\end{equation}
We then define 
$$\psi_{\gamma}(\alpha,\lambda,v):=(\lambda-R_{\gamma}(\alpha,\lambda,v), v-S_{\gamma}(\alpha,\lambda,v)).$$ Then the solutions of equation (\ref{bifurcation3}) are the zeros of $\psi_{\gamma}$ and conversely. The trivial solutions of (\ref{bifurcation3}) correspond to the solutions $(0,\lambda,0),~\lambda\in\R$ of equations (\ref{lss2}) and (\ref{lss3}). Since $\psi_{\gamma}$ is continuous and compact with respect to $(\lambda,v)$ for fixed $\alpha>0$ , for fixed $\alpha $ it has the appropriate form for the use of the theory of Leray-Schauder degree.
\subsection{Bifurcation from trivial solutions} Here, we prove our bifurcation result.\\
Let $\sigma(-\Delta-\gamma(\cdot)^-):=\{\lambda\in \R:~\exists u\neq 0~\text{in}~L^2(\Omega):~-\Delta u-\gamma(u)^-=\lambda u\}. $
\begin{theorem}\label{bifurcationtri}
There exists a solution branch $\mathcal{C}_{\lambda^1_k}$ in $\Sigma_{\gamma}$ such that $(\lambda^1_k,0)\in \mathcal{C}_{\lambda^1_k}.$ Then $\mathcal{C}_{\lambda^1_k}$ either
\begin{enumerate}
    \item[(a)] meets infinity in $\R\times L^2(\Omega),$ or
    \item[(b)] meets $(\mu,0),$ where $\mu\in \sigma(-\Delta-\gamma(\cdot)^-)\backslash\{\lambda^1_k,\lambda^2_k\}.$
\end{enumerate}
The analogous statement holds for $\mathcal{C}_{\lambda^2_k}.$
\end{theorem}
\noindent This bifurcation result is explained by Figure \ref{fig} below.
\begin{figure}	[!h]\centering
	\includegraphics[width=14cm,height=8cm]{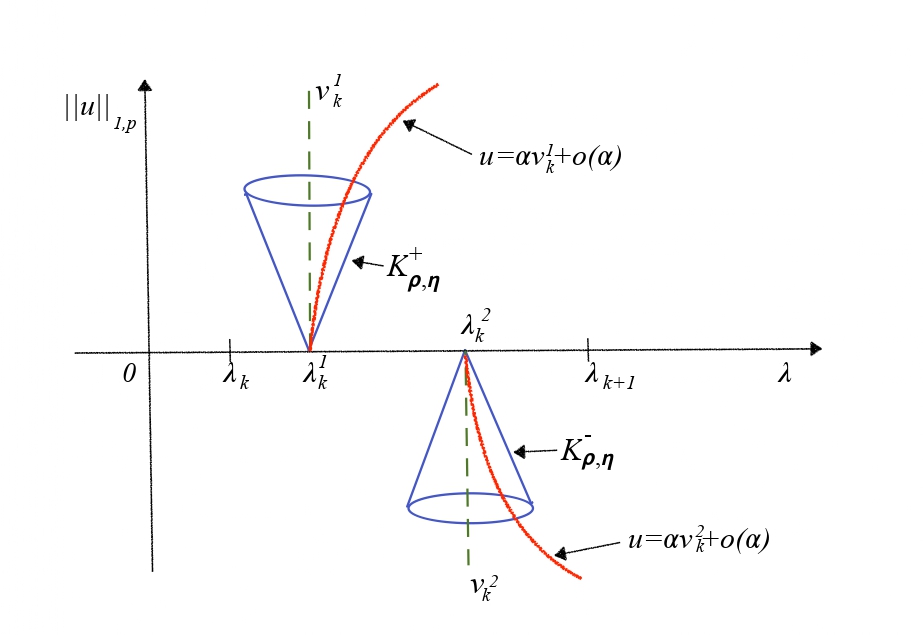}
	\caption{Bifurcation from trivial solutions.} \label{fig}
\end{figure}
	%\end{center}
	\newpage
	\begin{proof}
Let $\mathscr{U}_{\gamma}=\{(\lambda,v)\in \R\times H_k~/|\lambda-\lambda^1_k(\gamma)|<\varepsilon_1,~~\varrho_1>\|v\|\},$ where $\gamma\in [0,\bar{\gamma}]$ with $\bar{\gamma}$ arbitrary but fixed and satisfying $0<\bar{\gamma}<\lambda_{k+1}-\lambda_k.$ By Lemma \ref{l1} all the nontrivial solutions $(\alpha, \lambda, v)$ of equations (\ref{ls2}) and  (\ref{lss2}) near $(0,\lambda^1_k,0)$ satisfy $(\lambda,v)\in\mathscr{U}_{\gamma}$ if $|\alpha|<\alpha_1$ provided that $\alpha_1,\varepsilon_1,\varrho_1$ are sufficiently small and $\gamma\in [0,\bar{\gamma}].$ Therefore $\psi_{\gamma}(\alpha,\lambda,v)\neq 0$ on $\partial\mathscr{U}_{\gamma}$ for $0<\alpha<\alpha_1$ and $\gamma\in [0,\bar{\gamma}].$ Consequently the Leray-Schauder degree $\deg(\psi_{\gamma}(\alpha,\lambda,v), \mathscr{U}_{\gamma},(0,0))$ is well defined for $0<|\alpha|<\alpha_1.$ Since $\lambda^1_k$ is continuously dependent on $\gamma$ (see \cite[Lemma 2.4]{Ruf}, $\bigcup\limits_{[0,\bar{\gamma}]}\{\gamma\}\times \mathscr{U}_{\gamma}$ is a bounded open sets in $[0,\bar{\gamma}]\times\R\times H_k$. By the homotopy invariance of the Leray-Schauder degree, one has 
$$\deg(\psi_{\gamma}(\alpha,\lambda,v), \mathscr{U}_{\gamma},(0,0))=c_{\pm},~~\text{for}~~0<|\alpha|<\alpha_1.$$ To evaluate $c_{\pm}$, we consider the family of operators $\Theta_{\tau}$ defined as
$$\Theta_{\tau}(\alpha,\lambda,v):=\left(\lambda-(\frac{\alpha}{\lambda_k}+1)\lambda-\frac{\tau}{\lambda_k}(Q_k(T_{\lambda}(\alpha e_k+v)), e_k)_2+\alpha, v-\tau S_{\gamma}(\alpha,\lambda,v)\right)$$ for $\tau\in [0,1]$. As above, we can assume that $\Theta_{\tau}(\alpha,\lambda,v)\neq 0$ on $\partial\mathscr{U}_{\gamma}$ for $\tau\in [0,1]$ and $\alpha\in (-\alpha_1,\alpha_1).$ Using again the homotopy invariance of the degree implies that
\begin{equation}\label{degree}
    \deg(\Theta_{\tau}(\alpha,\lambda,v), \mathscr{U}_{\gamma},(0,0))=c_{\pm},~~\text{for}~~0<|\alpha|<\alpha_1~~\text{and}~~\tau\in[0,1].
\end{equation}
Therefore to evaluate $c_{\pm},$ it suffices to take $\tau=0,$ and 
$$\Theta_{0}(\alpha,\lambda,v)=\left(\lambda-(\frac{\alpha}{\lambda_k}+1)\lambda+\alpha, v\right).$$ The only zero that $\Theta_0$ possesses is $v=0,$ $\lambda=\lambda_k.$ We note that $\Theta_{0}(\alpha,\lambda,v)$ is an isomorphism on $\R\times H_k,$ and hence
$$c_{\pm}=\text{ind}(\Theta_{0}(\alpha,\lambda,v), (\lambda_k,0))=\pm 1.$$ We observe that 
\begin{equation}\label{ls3}
    \text{ind}(\Theta_{0}(\alpha,\lambda,v), (\lambda_k,0))=\deg(\hat{\Theta}_{0}(\alpha,\lambda,v), \mathscr{U}_{\gamma},(0,0))=\deg(\Theta_{0}(\alpha,\lambda,v), \mathscr{U}_{\gamma},(-\alpha,0)),
\end{equation}
where $\hat{\Theta}_{0}(\alpha,\lambda,v)=\left(\lambda-(\frac{\alpha}{\lambda_k}+1)\lambda+\alpha, v\right),$ which is the homogeneous (in $\alpha$) linear operator corresponding to $\Theta_{0}(\alpha,\lambda,v)$. For $\alpha\neq 0,$ $\hat{\Theta}_{0}(\alpha,\lambda,v)$ is an isomorphism and then
\begin{equation}\label{ls4}
   \deg(\hat{\Theta}_{0}(\alpha,\lambda,v), \mathscr{U}_{\gamma},(-\alpha,0))= \text{ind}(\hat{\Theta}_{0}(\alpha,\lambda,v), (-\alpha,0))=\text{ind}(\hat{\Theta}_{0}(\alpha,\lambda,v), (0,0)).
\end{equation}
Since $\left(\lambda-\gamma(\frac{\alpha}{\lambda_k}+1)\lambda+\alpha, v\right)$ has no characteristic values $\gamma$ in $(0,1)$ for $\alpha<0$ and one characteristic value in the interval $(0,1)$ for $\alpha>0,$ the theorem on change on index and relation (\ref{degree}) and (\ref{ls4}) imply that $c_{\pm}=\pm1.$ On the other hand, by Lemma \ref{l1} we have that
$$(-\Delta_p-\Delta)(\alpha e_k+v)-\bar{\gamma}(\alpha e_k+v)^-=\lambda (\alpha e_k+v),~~v\in H_k$$ has no solution for $-\alpha_1<\alpha<0$ and $|\lambda-\lambda^1_k(\bar{\gamma})|<\rho_0.$ Since this equation is equivalent to equation (\ref{ls2}), we infer to
$$\deg(\psi_{\bar{\gamma}}(\alpha,\lambda,v), \mathscr{U}_{\bar{\gamma}},(0,0))=0,~~\text{for}~~-\alpha<\alpha<0.$$ With this jump in the degree, we obtain the existence of the solution branch $\mathcal{C}_{\lambda^1_k}.$ The branch $\mathcal{C}_{\lambda^2_k}$ is obtained similarly.
	\end{proof}
\subsection{Bifurcation from infinity}\label{Sinf} Here, we prove a  bifurcation from infinity result. As in \cite{ZB}, we introduce the following change of variable:
for $u\in W^{1,2}_0(\Omega),~u\neq 0,$ we set $v=u/\|u\|_{1,2}^{2-\frac{1}{2}p}$ for all $1<p<2.$  We have $\|v\|_{1,2}=\frac{1}{\|u\|_{1,2}^{1-\frac{1}{2}p}}$ and 
$$|\nabla v|^{p-2}\nabla v=\frac{1}{\|u\|_{1,2}^{(2-\frac{1}{2}p)(p-1)}} |\nabla u|^{p-2}\nabla u.$$
Introducing this change of variable in (\ref{weakform}), we find that, 
\begin{equation}\label{em3}
\|u\|_{1,2}^{(2-\frac{1}{2}p)(p-2)}\int_{\Omega}|\nabla v|^{p-2}\nabla v\cdot\nabla\varphi~dx+\int_{\Omega}\nabla v\cdot\nabla\varphi~dx=\lambda_+\int_{\Omega}v^+\varphi~dx-\lambda_-\int_{\Omega}v^-\varphi~dx
\end{equation}
for every $\varphi\in W^{1,2}_0(\Omega)$.
But, on the other hand, we have
$$\|v\|^{p-4}_{1,2}=\frac{1}{\|u\|_{1,2}^{(1-\frac{1}{2}p)(p-4)}}=\frac{1}{\|u\|_{1,2}^{(2-\frac{1}{2}p)(p-2)}}.$$ Consequently it follows that equation (\ref{em3}) is equivalent to
\begin{equation}\label{em4}
\|v\|_{1,2}^{4-p}\int_{\Omega}|\nabla v|^{p-2}\nabla v\cdot\nabla\varphi~dx+\int_{\Omega}\nabla v\cdot\nabla\varphi~dx=\lambda_+\int_{\Omega}v^+\varphi~dx-\lambda_-\int_{\Omega}v^-\varphi~dx,
\end{equation}
for every $\varphi\in W^{1,2}_0(\Omega)$.
This leads to the following nonlinear eigenvalue problem ($1<p<2$)
\begin{equation}\label{em5}
\left\{
\begin{array}{l}
-\|v\|_{1,2}^{4-p}\Delta_p v-\Delta v =\lambda_+v^+-\lambda_-v^-~~\text{in $\Omega$},\\
v = \displaystyle 0~~~~~~~~~~~~~~~~\text{on $\partial\Omega$}.
\end{array}
\right.
\end{equation}
\begin{remark}\label{binfty}
With this transformation, we have that the pair $(\lambda^1_k,\infty)$ or $(\lambda^2_k,\infty)$ is a bifurcation point for the problem (\ref{equa3}) if and only if the pair $(\lambda^1_k,0)$ or $(\lambda^2_k,0)$ is a bifurcation point for the problem (\ref{em5}).
\end{remark}
For $v\in W^{1,2}_0(\Omega),$ we have $$-\|v\|_{1,2}^{4-p}\Delta_p v-\Delta v=\lambda_+v^+-\lambda_-v^-\Leftrightarrow -\|v\|_{1,2}^{4-p}\Delta_p v-\Delta v-(\lambda_+-\lambda_-)v^-=\lambda_+v.$$ We fix $\gamma=\lambda_+-\lambda_->0$ and we consider $\lambda:=\lambda_+$ as a parameter of the equation
\begin{equation}\label{bifurcationinf1}
  \begin{cases}
    -\|v\|_{1,2}^{4-p}\Delta_p v-\Delta v-\gamma v^-=\lambda v ~~~\text{on}~~\Omega,\\
    v=0~~~~~~~~~~~\text{on}~~\partial\Omega.
  \end{cases}
\end{equation}
Let us consider a small ball $B_r :=\{~w~\in W^{1,2}_0(\Omega):~~~\|w\|_{1,2}< r~\},$ and
consider the operator $$A :=-\|\cdot\|_{1,2}^{4-p}\Delta_p-\Delta-\gamma(\cdot)^- : W^{1,2}_0(\Omega)\rightarrow W^{-1,2}(\Omega).$$
\begin{proposition}\label{invertb}
Let $1<p<2.$ There exists $r>0$ such that the mapping \\$A : B_r(0)\subset W^{1,2}_0(\Omega)\rightarrow W^{-1,2}(\Omega)$ is invertible, with a continuous inverse.
\end{proposition}
We first remark that for $\gamma>0,$ $-\gamma(u^--v^-,u-v)_2\geq 0.$ Indeed, we have 
\begin{equation*}
    \begin{split}
        (u^--v^-,u-v)_2&=\int_{\Omega}(u^--v^-)(u-v)dx\\
        &=\int_{\Omega}(u^--v^-)(u^+-v^+)dx-\int_{\Omega}|u^--v^-|^2dx\\
        &=\frac{1}{4}\left[(|u|-|v|)^2-\|u-v\|^2_2\right]-\|u^--v^-\|^2_2\leq 0,
    \end{split}
\end{equation*}
using the definition of $u^{\pm}$ and $v^{\pm}$.
\begin{proof}[\textup{Proof of Proposition \ref{invertb}}]
In order to prove that the operator $A$ is invertible with a continuous inverse, it is enough to prove that $A$ is continuous strongly monotone operator. 
We show that there exists $\delta>0$ such that $$\langle A(u)-A(v), u-v\rangle\geq \delta\|u-v\|^2_{1,2}, ~~\text{for}~~u,v\in B_r(0)\subset W^{1,2}_0(\Omega)$$ with $r>0$ sufficiently small.\\
Indeed, using that $-\Delta_p$ is strongly monotone on $W^{1,p}_0(\Omega)$ on the one hand and the H\"older inequality on the other hand, we have
\begin{eqnarray}\label{bf1}\notag
\langle A(u)-A(v), u-v\rangle&=&\|\nabla u-\nabla v\|^2_{2}+\left(\|u\|_{1,2}^{4-p}(-\Delta_pu)-\|v\|_{1,2}^{4-p}(-\Delta_pv), u-v \right)_2\\ \notag &-&\gamma(u^--v^-,u-v)_2\\ \notag
&=&\|u-v\|^2_{1,2}+\|u\|_{1,2}^{4-p}\left((-\Delta_pu)-(-\Delta_p v), u-v\right)_2\\ \notag
&+& \left(\|u\|_{1,2}^{4-p}-\|v\|_{1,2}^{4-p} \right)\left(-\Delta_pv, u-v\right)_2-\gamma(u^--v^-,u-v)_2\\ \notag
&\geq & \|u-v\|^2_{1,2}-\left|\|u\|_{1,2}^{4-p}-\|v\|_{1,2}^{4-p}\right|\|\nabla v\|_{p}^{p-1}\|\nabla (u-v)\|_{p}\\
&\geq& \|u-v\|^2_{1,2}-\left|\|u\|_{1,2}^{4-p}-\|v\|_{1,2}^{4-p}\right|C\| v\|_{1,2}^{p-1}\| u-v\|_{1,2}.
\end{eqnarray}
Now, we obtain by the Mean Value Theorem  that there exists $\theta\in [0,1]$ such that
\begin{eqnarray*}
\left|\|u\|_{1,2}^{4-p}-\|v\|_{1,2}^{4-p}\right|&=&\left|\frac{d}{dt}\left(\|u+t(v-u)\|^2_{1,2}\right)^{2-\frac{1}{2}p}|_{t=\theta} (v-u)\right|\\
&=&\left|(2-\frac{1}{2}p)\left(\|u+\theta(v-u)\|^2_{1,2}\right)^{1-\frac{1}{2}p}2\left(u+\theta(v-u),v-u\right)_{1,2}\right|\\
&\leq& (4-p)\|u+\theta(v-u)\|^{2-p}_{1,2}\|u+\theta(v-u)\|_{1,2}\|u-v\|_{1,2}\\
&=&(4-p)\|u+\theta(v-u)\|_{1,2}^{3-p}\|u-v\|_{1,2}\\
&\leq & (4-p)\left((1-\theta)\|u\|_{1,2}+\theta\|v\|_{1,2}\right)^{3-p}\|u-v\|_{1,2}\\
&\leq & (4-p) r^{3-p}\|u-v\|_{1,2}.
\end{eqnarray*}
Hence, continuing with the estimate of equation (\ref{bf1}), we get
$$\langle A(u)-A(v), u-v\rangle\geq \|u-v\|^2_{1,2}(1-(4-p) r^{3-p}Cr^{p-1})=\|u-v\|^2_{1,2}(1-C'r^{2}),$$ and thus the claim, for $r>0$ small enough. Hence, the operator $A$ is strongly monotone on $B_r(0)$ and it is continuous, and hence the claim follows.
\end{proof}
Using Proposition \ref{invertb}, we write equation (\ref{bifurcationinf1}) as the following form:
\begin{equation}\label{bifurcationinf2}
    v=\lambda(-\Delta)^{-1}v+\left\{(-\|\cdot\|_{1,2}^{4-p}\Delta_p-\Delta-\gamma(\cdot)^-)^{-1}-(-\Delta)^{-1}\right\}(\lambda v),
\end{equation}
where we consider
$$(-\|\cdot\|_{1,2}^{4-p}\Delta_p-\Delta-\gamma(\cdot)^-)^{-1}: L^2(\Omega)\rightarrow W^{1,2}_0(\Omega)\subset\subset L^2(\Omega),$$
and 
$$(-\Delta)^{-1}: L^2(\Omega)\subset W^{-1,2}(\Omega)\rightarrow W^{1,2}_0(\Omega)\subset\subset L^2(\Omega).$$
So, for $1<p<2$ the mapping $$(-\|\cdot\|_{1,2}^{4-p}\Delta_p-\Delta-\gamma(\cdot)^-)^{-1}-(-\Delta)^{-1}: L^2(\Omega)\subset W^{-1,2}(\Omega)\rightarrow W^{1,2}_0(\Omega)\subset\subset L^2(\Omega)$$ is a compact mapping thanks to the Rellich-Kondrachov theorem.\\
\\
For $\gamma>0$ fixed, let $\mathcal{P}_{\gamma}$ denote the closure of the set of nontrivial solutions of problem (\ref{bifurcationinf1}).
The following lemma localizes the possible solutions of equation (\ref{bifurcationinf1}).
\begin{lemma}\label{linf1}
There exists $\rho_0>0$ such that for all $0<\rho<\rho_0:$
$$\left(\mathcal{P}_{\gamma}\backslash\{(\lambda^1_k,0)\}\cap B_{\rho}\right)\subset \pmb{K^+_{\rho,\eta}}(\gamma).$$ If $(\lambda, u)\in \left(\mathcal{P}_{\gamma}\backslash\{(\lambda^1_k,0)\}\cap B_{\rho}\right),$ then $u=\alpha e_k+v,$ where $|\alpha|>\eta \|u\|_2,$ and $|\lambda-\lambda^1_k|=o(1), v=o(1)$ for $\alpha$ near zero. Moreover $\alpha>0.$ The same statement hold for $\pmb{K^-_{\rho,\eta}}(\gamma).$
\end{lemma}
The proof of this lemma is similar to the proof of Lemma \ref{l1} and it is left to the reader.
\begin{theorem}\label{infe}
There exists a solution branch $\mathcal{S}_{\lambda^1_k}$ in $\mathcal{P}_{\gamma}$ such that $(\lambda^1_k,0)\in \mathcal{S}_{\lambda^1_k}.$ Then $\mathcal{S}_{\lambda^1_k}$ either
\begin{enumerate}
    \item[(a)] meets infinity in $\R\times L^2(\Omega),$ or
    \item[(b)] meets $(\mu,0),$ where $\mu\in \sigma(-\Delta-\gamma(\cdot)^-)\backslash\{\lambda^1_k,\lambda^2_k\}.$
\end{enumerate}
The analogous statement holds for $\mathcal{S}_{\lambda^2_k}.$
\end{theorem}
One can adapt the proof of Theorem \ref{bifurcationtri} to this theorem. Theorem \ref{infe} can be illustrate as in Figure \ref{fig}, where the norm $\|u\|_{1,p}$ in the $y$-axis has to be replaced by $\|u\|_{1,2}$.\\
\\
%\newpage
\textbf{Acknowledgments}\\
\\
The author wish to thank Bernhard Ruf for many fruitful discussions on this work.
%\newpage

\end{document}